\documentclass[a4paper,oneside,10pt]{article}%
\usepackage{amsmath}
\usepackage{amsfonts}
\usepackage{amssymb}
\usepackage{graphicx}
\usepackage{xcolor}
\usepackage[square,numbers,sort&compress]{natbib}%
\usepackage[toc]{appendix}
\providecommand{\U}[1]{\protect\rule{.1in}{.1in}}

\newtheorem{theorem}{Theorem}[section]

\newtheorem{definition}[theorem]{Definition}
\newtheorem{assumption}[theorem]{Assumption}

\newtheorem{lemma}[theorem]{Lemma}

\newtheorem{proposition}[theorem]{Proposition}
\newtheorem{remark}[theorem]{Remark}

\newenvironment{proof}[1][Proof]{\noindent\textbf{#1.} }{\ \rule{0.5em}{0.5em}}
\numberwithin{equation}{section}

\allowdisplaybreaks[2]

\begin{document}

\title{The perturbation method applied to a robust optimization problem with constraint}
\author{Peng Luo\thanks{School of Mathematical Sciences, Shanghai Jiao Tong University, Shanghai 200240, China. E-mail: {\tt peng.luo@sjtu.edu.cn}}
	\and Alexander Schied\thanks{Department of Statistics and Actuarial Science, University of Waterloo, ON, N2L 3G1, Canada. E-mail: {\tt aschied@uwaterloo.ca} }
	\and Xiaole Xue\thanks{School of Management, Shandong University, Jinan 250100, China. E-mail: {\tt xlxue@sdu.edu.cn(Corresponding author) }\hfill\break The authors gratefully acknowledge financial support from the Natural Sciences and Engineering Research Council of Canada through grant RGPIN-2017-04054. Peng Luo gratefully acknowledges the support from the National Natural Science
		Foundation of China (Grant No. 12101400). Xiaole Xue gratefully acknowledges the support from National Natural Science Foundation of China  (No.12001316), ``The
		Fundamental Research Funds of Shandong University".}
}
\maketitle

\textbf{Abstract}. {The present paper studies a kind of robust optimization problems with constraint. The problem is formulated through Backward Stochastic Differential Equations (BSDEs) with quadratic generators. A necessary condition is established for the optimal solution using a terminal perturbation method and  properties of Bounded Mean Oscillation (BMO) martingales. The necessary condition is further proved to be sufficient for the existence of an optimal solution under an additional convexity assumption. Finally, the optimality condition is applied to discuss problems of partial hedging with ambiguity, fundraising under  ambiguity and randomized testing problems for a quadratic $g$-expectation.}


\textbf{AMS subject classifications.} 93E20, 60H10, 35K15\addcontentsline{toc}{section}{\hspace*{1.8em}Abstract}

\section{Introduction}
In financial markets, an investor usually needs to maximize over a set of prior measures and then minimize a cost criterion. This can be formulated as a robust optimization problem as follows
\begin{equation}\label{1st eq}
\inf_{\xi}\sup_{Q} U(\xi, Q)
\end{equation}
where $U(\xi,Q)$ is a cost functional based on an expectation under the probability measure $Q$.
There are two main ways to deal with this robust 
optimization problem. One is based on duality techniques, the other one  on stochastic optimal control. While the latter method typically requires the problem \eqref{1st eq} to be time-consistent, it often leads to more concrete characterizations of solutions, which are suitable for numerical computations. 

A  finite-horizon robust control criterion was characterized by Skiadas  \cite{Skiadas03} via a backward stochastic differential equation (BSDE).  Bordigoni et al.  \cite{Bordogoni07} proved the existence of a unique $Q^{*}$, and the value function is characterized by the unique solution of a generalized 
BSDE with a quadratic generator. In  \cite{Faidi2011},  Faidi et. al. studied a maximization problem for terminal wealth and consumption for a class of robust utility functions. The dynamic maximum principle for the optimal control is studied, and the existence and the uniqueness of the consumption-investment strategy is characterized by the unique solution of a forward-backward system.   Schroder and Skiadas  
\cite{SchroderSkiadas03} studied the lifetime  consumption-portfolio recursive utility problem under convex trading 
constraints by the utility gradient approach. 
The first-order conditions of optimality are obtained through a constrained forward backward stochastic differential equation. 
Recently, Faidi et al.  \cite{Faidi2019} studied a general robust utility maximization problem under terminal wealth 
and consumption with state constraints.
The existence and the uniqueness of the optimal strategy are obtained by 
studying the associated quadratic BSDE.  Additional references and a survey of related problems can be found in  Sections 8 and 9 of \cite{FoellmerSchiedWeber}.

The concept of $g$-expectations, which was first introduced by Peng \cite{peng1997}, has some intriguing relationship with various  kinds of risk measures
(see  \cite{chenap, gianin, hehuchen, jiang} for example).
The
$g$-expectation has been an important tool 
in economics, finance and insurance; see, e.g., \cite{bernard, cheneco,chenime, yongzhou} and the references therein.
In particular,  under some suitable conditions,
it is shown in \cite{karoui} that the following objective function
\[
\sup\limits_{Q\ll P} E^{Q}[f(\xi)-H(Q|P)]
\]
equals   $\mathcal{E}_{g}[f(\xi)]$,
where $H(Q|P)$ is the relative entropy and $\mathcal{E}_{g}[\cdot]$ is a quadratic $g$-expectation (the definition of quadratic $g$-expectations will be recalled in Section 2).
Thus, the robust optimization problem related to the above objective function can be characterized as
$$ \inf_{\xi}\mathcal{E}_{g}[f(\xi)].$$ 
With different constraints on $\xi$,  optimization problems of this kind arise naturally in applications to  finance and statistics such as partial hedging, liability minimization, and  randomized test theory for composite hypotheses.

Under an convexity condition, El Karoui, Peng and Quenez \cite{karoui2001} applied  the penalty method  and obtained a maximum principle, which is a necessary condition for optimality.  Ji and Peng \cite{jipeng}  first introduced the terminal perturbation method to a problem without convexity conditions and obtained an elegant result  by  employing Ekeland's variational principle. 
Ji and Zhou \cite{jizhou} applied this method  to  a hypothesis testing problem and obtained a Neyman--Pearson-type lemma  for $g$-probabilities.
We also refer to Bernard, Ji, and Tian \cite{bernard} and Ji \cite{Ji-jmaa} for respective applications of the perturbation method in insurance and mean-variance portfolio. 

In this paper, we study a  robust optimization problem with constraints. The problem is formulated  through BSDEs with quadratic generators,  and the objective is to find a terminal condition    minimizing   a quadratic $g$-expectation subject to a specific constraint. More precisely, 
the robust optimization problem with constraint is formulated as follows:
\begin{equation*}
\left\{
\begin{aligned} \inf\limits_{\xi\in U} & \ \mathcal{E}_{f}[h(\xi)+\alpha\mathcal{E} _{g}[\xi]]\\ s.t. & \ \mathcal{E}_{g}[\xi]\leq\pi_{0}, \end{aligned} \right.
\end{equation*}
where  $\mathcal{E}_{f}[\cdot]$ and  $\mathcal{E}_{g}[\cdot]$ are quadratic $g$-expectations, $\alpha$ and $\pi_{0}$ are two constants,  and the admissible set $U$ is given by
\[
U=\{\xi\in L^{\infty}(\mathcal{F}_{T});X\leq\xi\leq Y\}
\]
with $X, Y\in L^{\infty}(\mathcal{F}_{T})$ and $X\leq Y$.   

 We use the terminal perturbation to obtain the necessary conditions for  solutions of the above optimization problems. In 	contrast to the case of generators satisfying a  linear growth condition, as discussed in  Ji and Peng \cite{jipeng} and  Ji and Zhou \cite{Ji-Zhou-PT}, we only require a quadratic growth condition. This relaxed assumption creates additional difficulties, for instance through the need of obtaining a priori estimates for variational equations.  These difficulties are overcome by employing some delicate properties of  BMO martingales.  To our best knowledge, this is the first  study of the  robust optimization problem  using  the terminal  perturbation method for  quadratic generators. 
We characterize  the optimal strategy  by  a forward-backward system. Under an additional convexity assumption, we furthermore show that our above-mentioned necessary condition is also sufficient for the existence of an optimal solution. Moreover,  our results have applications in finance, actuarial science and statistics. This is illustrated by concrete applications to a  partial hedging problem,  a  fundraising problem under ambiguity, and the derivation of  a Neyman--Pearson lemma for quadratic $g$-probabilities.

The paper is organized as follows. In section 2, we present our proposed robust optimization problem with constraint, necessary and sufficient conditions are obtained. Section 3 is devoted to present some concrete applications.

\section{Problem formulation} \label{sec:problem_formulation}

\subsection{Notation}
Let $W=(W_{t})_{t\geq0}$ be a $d$-dimensional Brownian motion on a probability space $(\Omega, \mathcal{F}, 
P)$, and denote by $\{\mathcal{F}_{t}\}_{t\geq0}$ the complete filtration generated by $W$. Throughout the paper, 
we fix a finite time horizon $[0, T]$ with $T>0$. We endow $\mathbb{R}^{n}$ with its Borel $\sigma$-algebra $
\mathcal{B}(\mathbb{R}^{n})$.   For an event $A\in\mathcal{F}$, $I_A$ denotes a Bernoulli random variable 
with $I_A(\omega)=1$ for $\omega\in A$ and $I_A(\omega)=0$ otherwise. Equalities and inequalities between 
random variables and processes are in the sense of $P$-a.s. and $P\otimes dt$-a.e., respectively. The 
Euclidean norm is denoted by $|\cdot|$. For $m\in[1,\infty]$, we denote by $\|\cdot\|_{m}$ the usual $L^{m}$-norm, by 
$L^{m}(\mathcal{F}_{t})$ the set of real-valued $\mathcal{F}_{t}$-measurable random variables $X$ such that $\|X\|_{m}<\infty$. 
By   $\mathcal{S}^{m}$ we denote the set of real-valued continuous $\{ \mathcal{F}_{t}\}_{t\geq0}$-adapted 
processes $Y$ on $[0,T]$ such that
\[
\|Y\|_{\mathcal{S}^{m}}:=\left\| \sup_{0\leq t\leq T} |Y_{t}|\right\| _{m}<
\infty .
\]
By $\mathcal{H}^{m}$ we denote the set of $\{ \mathcal{F}_{t}\}_{t\geq0}$-adapted $\mathbb{R}^{d}$-valued processes $Z$ such that
\[
\|Z\|_{\mathcal{H}^{m}}=\left\| \left( \int_{0}^{T}|Z_{s}|^{2}ds\right)
^{\frac{1}{2}}\right\| _{m}<\infty .
\]
The space $\text{\textrm{BMO}}$ consists of all $\{ \mathcal{F}_{t}\}_{t\geq 0}$-adapted $\mathbb{R}^{d}$-valued processes $Z$ such that
\[
\|Z\|_{\text{\textrm{BMO}}}=\sup_{\tau\in\mathcal{T}}\Bigg\|E\left[ \left(
\int_{\tau}^{T}|Z_{s}|^{2}ds\right) ^{\frac{1}{2}}\Bigg|\mathcal{F}_{\tau
}\right] \Bigg\|_{\infty}<\infty,
\]
where $\mathcal{T}$ is the set of all stopping times with values in $[0,T]$.

\subsection{ Assumptions and model formulation}

Our constrained optimization problem is formulated through $g$-expectations and quadratic BSDEs. Before introducing the specific formulation of the problem, we first introduce some assumptions on the generator functions of the BSDEs and the definition of the $g$-expectation.

\begin{assumption}
	\label{ass-qua} The function $g:\Omega\times[0,T]\times\mathbb{R}\times\mathbb{R}^{d}\rightarrow\mathbb{R}$ is progressively measurable and $g(\omega,t,\cdot,\cdot)$ is continuously differentiable for all $(\omega,t)\in\Omega\times[0,T]$. Moreover, there exists a constant $C>0$ such that for all $t\in[0,T]$ and any $y,\bar{y}\in\mathbb{R},z,\bar{z}\in\mathbb{R}^{d}$,
	\begin{align*}
	& |g(t,y,z)|\leq C(1+|y|+|z|^{2}),~~|g_{y}(t,y,z)|\leq C, ~~|g_{y}%
	(t,y,z)-g_{y}(t,\bar{y},\bar{z})|\leq C(|y-\bar{y}|+|z-\bar{z}|),\\
	& ~~~~|g_{z}(t,y,z)|\leq C(1+|z|),~~\text{and}~~|g_{z}(t,y,z)-g_{z}(t,\bar
	{y},\bar{z})|\leq C(|y-\bar{y}|+|z-\bar{z}|).
	\end{align*}
	
\end{assumption}

\begin{assumption}
	\label{ass-ini} The function $g$ satisfies $g(t,y,0)=0$ $P$-a.s., for all $t\in[0,T]$, $ y\in\mathbb{R}$.
\end{assumption}

\begin{assumption}
	\label{ass-h} The map $h:\Omega\times\mathbb{R}\rightarrow\mathbb{R}$ is  such that for all $x\in\mathbb{R}$, $h(\cdot,x)$ is $\mathcal{F}_{T}$-measurable and for all $\omega\in\Omega$, $h(\omega,\cdot)$ is continuously differentiable.  
	There exists a constant $C>0$ such that $|h(0)|\leq C$,  $|h_{x}(x)|\leq C$, and $|h_{x}(x)-h_{x}(y)|\leq C|x-y|$  for all $x,y\in\mathbb{R}$.
\end{assumption}

\begin{lemma}
	[\cite{briand}]\label{exist_lemma}
	For a function $g$ satisfying Assumption \ref{ass-qua} and a random variable  $\xi\in L^{\infty}(\mathcal{F}_{T})$,   the BSDE
	\begin{equation}
	\label{bsde1}y(t)=\xi+\int_{t}^{T}g(s,y(s),z(s))ds-\int_{t}^{T}z(s)dW_{s}%
	\end{equation}
	admits a unique solution $(y,z)\in\mathcal{S}^{\infty}\times\text{\textrm{BMO}}$. Moreover, there exists a constant $L$ that only depends on $\xi$ and  the constant $C$ from Assumption~\ref{ass-qua} such that
	\[
	\|y\|_{\mathcal{S}^{\infty}}+\|z\|_{\text{\textrm{BMO}}}\leq L.
	\]
	
\end{lemma}

\begin{definition}
	[\cite{ma}]
	Suppose that $g$ satisfies Assumptions \ref{ass-qua} and \ref{ass-ini}. Given $\xi\in L^{\infty}(\mathcal{F}_{T})$, let $(y(\cdot),z(\cdot))\in\mathcal{S}^{\infty}\times\text{\textrm{BMO}}$ be the unique solution of the BSDE (\ref{bsde1}). Then the  $g$-expectation of $\xi$ is defined by
	\[
	\mathcal{E}_{g}[\xi]:=y(0).
	\]
	
\end{definition}

Let $f$ and $g$ be two functions satisfying Assumptions \ref{ass-qua} and \ref{ass-ini} and $h$  be  another function
fulfilling Assumption \ref{ass-h}.  The constrained optimization problem we study in the paper is formulated as follows:
\begin{equation}
\label{g-object}\left\{
\begin{aligned} \inf\limits_{\xi\in U} & \ \mathcal{E}_{f}[h(\xi)+\alpha\mathcal{E} _{g}[\xi]]\\ s.t. & \ \mathcal{E}_{g}[\xi]\leq\pi_{0}, \end{aligned} \right.
\end{equation}
where $\alpha$ and $\pi_{0}$ are two constants,  and the admissible set $U$  is given by
\[
U=\{\xi\in L^{\infty}(\mathcal{F}_{T});X\leq\xi\leq Y\}
\]
with $X, Y\in L^{\infty}(\mathcal{F}_{T})$ and $X\leq Y$.   With these assumptions imposed on $f$, $g$ and $h$, the problem \eqref{g-object} is well defined, since Lemma \ref{exist_lemma} ensures that for any $\xi\in L^{\infty}(\mathcal{F}_{T})$, the  BSDEs
\begin{equation}
\label{bsde object}
\left\{
\begin{aligned}
dy_{1}(t)&=-f(t,y_{1}(t),z_{1}(t))dt+z_{1}(t)dW_{t},\\ y_{1}(T)&=h(\xi)+\alpha\mathcal{E}_{g}[\xi],\\ dy_{2}(t)&=-g(t,y_{2}(t),z_{2}(t))dt+z_{2}(t)dW_{t},\\ y_{2}(T)&=\xi,
\end{aligned}
\right.
\end{equation}
admit a unique solution $(y_{1}, z_{1}, y_{2}, z_{2})\in\mathcal{S}^{\infty}\times\text{\textrm{BMO}}\times\mathcal{S}^{\infty}\times\text{\textrm{BMO}}$. Recalling the comparison theorem for quadratic BSDEs from  \cite[Theorem 2.6]{kobylanski},  we assume that $\mathcal{E}_{g}[X] < \pi_{0} <\mathcal{E}_{g}[Y]$ to avoid the special cases where the constraint in \eqref{g-object} becomes either invalid or unnecessary.
 In fact, the comparison theorem for quadratic BSDEs implies that, if $ \pi_{0}\leq \mathcal{E}_{g}[X]$, then there exists at most  one element (namely $X$) satisfying the constraint; if $ \pi_{0}\geq \mathcal{E}_{g}[Y]$, then all elements in $U$ satisfy the constraint, thus the constraint is immaterial.
\begin{remark}
	The optimization problem~\eqref{g-object} is quite general and encompasses  many models from different topics as special cases. In Section~\ref{sec:app}, we will discuss in detail its applications to topics of partial hedging, liability minimization, and randomized testing,   respectively.
\end{remark}

\subsection{Terminal perturbation method}

In this section, we  study the optimality condition on the optimal solution to problem (\ref{g-object}) by the terminal perturbation method. This method  was introduced by El Karoui, Peng and Quenez \cite{karoui2001} in order to solve an optimization problem for  recursive utility under some convexity conditions.  Later on, it was further developed  to study an insurance design problem \cite{bernard}, a mean-variance portfolio selection problem with non-convex wealth  \cite{Ji-jmaa,jipeng,jizhou}, and to prove a generalized Neyman--Pearson lemma \cite{Ji-Zhou-PT}. 

Suppose that  $\xi^{\ast}\in U$ is an optimal solution to the problem (\ref{g-object}). Further,  let $(y^{*}_{1}, z^{*}_{1}, y^{*}_{2}, z^{*}_{2})\in\mathcal{S}^{\infty}\times\text{\textrm{BMO}}\times\mathcal{S}^{\infty}\times\text{\textrm{BMO}}$ be the unique solution
of \eqref{bsde object} with $\xi=\xi^{*}$. Our first result provides a necessary condition for the optimal solution $\xi^{\ast}$.

\begin{theorem}
	\label{theorem}
	Suppose that $f$ and $g$ satisfy Assumptions~\ref{ass-qua} and \ref{ass-ini} and $h$ obeys Assumption~\ref{ass-h}. Suppose $\xi^{\ast}\in U$ is an optimal solution to (\ref{g-object}), then $\xi^{\ast}$ must be of the following form:
	\begin{align*}
	\xi^{\ast}= & YI_{\{h_{2}m(T)+h_{1}h_{x}(\xi^{\ast})n(T)+\alpha h_{1}
		m(T)E[n(T)]<0\}}+XI_{\{h_{2}m(T)+h_{1}h_{x}(\xi^{\ast})n(T)+\alpha
		h_{1}m(T)E[n(T)]>0\}}\\
	& +bI_{\{h_{2}m(T)+h_{1}h_{x}(\xi^{\ast})n(T)+\alpha h_{1}m(T)E[n(T)]=0\}},
	\end{align*}
	where $h_{1}\geq0$ and  $h_{2}\in\mathbb{R}$ satisfy  $h_1^2+h_2^2=1$, $b\in L^{\infty}(\mathcal{F}_{T})$ satisfies $X\leq b\leq Y$, and  $m(\cdot)$ and $n(\cdot)$ are given by the following adjoint equations:
	\begin{equation}
	\label{adjoint}\left\{
	\begin{aligned} dn(t)&=f_{y}(t,y_{1}^{\ast}(t),z_{1}^{\ast}(t))n(t)dt+f_{z}(t,y_{1}^{\ast }(t),z_{1}^{\ast}(t))n(t)dW_{t},\\ dm(t)&=g_{y}(t,y_{2}^{\ast}(t),z_{2}^{\ast}(t))m(t)dt+g_{z}(t,y_{2}^{\ast }(t),z_{2}^{\ast}(t))m(t)dW_{t},\\ n(0)&=1,\ \ m(0)=1. \end{aligned} \right.
	\end{equation}
	
\end{theorem}

Before proving Theorem \ref{theorem}, we  develop some technical lemmas. For each $\rho\in [0, \,1]$ and $\xi\in U$, we introduce $\xi^{\rho}=\xi^{\ast}+\rho(\xi-\xi^{\ast})$.  Let $(y_{1}^{\rho}, z_{1}^{\rho}, y_{2}^{\rho}, z_{2}^{\rho})\in\mathcal{S}^{\infty}\times\text{\textrm{BMO}}\times\mathcal{S}^{\infty}\times\text{\textrm{BMO}}$
be the unique solution to (\ref{bsde object})  with $\xi$ replaced by $\xi^{\rho}$. It is easy to verify that there exists a unique solution $  (\hat{y}_{1}, \hat{z}_{1}, \hat{y}_{2}, \hat{z}_{2})\in\mathcal{S}^{\infty}\times\text{\textrm{BMO}}\times\mathcal{S}^{\infty}\times\text{\textrm{BMO}}$ to the following variational equation:
\begin{equation}
\label{bsde variation}\left\{
\begin{aligned} -d\hat{y}_{1}(t)&=[f_{y}(t,y_{1}^{\ast}(t),z_{1}^{\ast}(t))\hat{y} _{1}(t)+f_{z}(t,y_{1}^{\ast}(t),z_{1}^{\ast}(t))\hat{z}_{1}(t)]dt-\hat{z}_{1}(t)dW_{t},\\ \hat{y}_{1}(T)&=h_{x}(\xi^{\ast})(\xi-\xi^{\ast})+\alpha\hat{y}_{2}(0),\\ -d\hat{y}_{2}(t)&=[g_{y}(t,y_{2}^{\ast}(t),z_{2}^{\ast}(t))\hat{y} _{2}(t)+g_{z}(t,y_{2}^{\ast}(t),z_{2}^{\ast}(t))\hat{z}_{2}(t)]dt-\hat{z}_{2}(t)dW_{t},\\ \hat{y}_{2}(T)&=\xi-\xi^{\ast}. \end{aligned} \right.
\end{equation}
For each $i=1,2$, we denote
\begin{equation*}
\tilde{y}_{i}^{\rho}(t)  =\frac{y_{i}^{\rho}(t)-y_{i}^{*}(t)}{\rho}-\hat
{y}_{i}(t),\quad \mbox{and}\quad  \tilde{z}_{i}^{\rho}(t)  =\frac{z_{i}^{\rho}(t)-z_{i}^{*}(t)}{\rho}-\hat
{z}_{i}(t).
\end{equation*}
\begin{lemma}
	\label{lemma estimation}
	Suppose that $f$ and $g$ satisfy Assumptions \ref{ass-qua} and \ref{ass-ini} and $h$ obeys Assumption \ref{ass-h}. Then \begin{align*}
	\lim_{\rho\rightarrow0} \|\tilde{y}_{1}^{\rho}\|_{\mathcal{S}^{\infty}}=0,~\lim_{\rho\rightarrow0} \|\tilde{y}_{2}^{\rho}\|_{\mathcal{S}^{\infty}}=0, ~\lim_{\rho\rightarrow0} \|\tilde{z}_{1}^{\rho}\|_{\text{\textrm{BMO}}}=0, \text{ and }\lim_{\rho\rightarrow0} \|\tilde{z}_{2}^{\rho}\|_{\text{\textrm{BMO}}} =0.
	\end{align*}
	
\end{lemma}

\begin{proof}
	From \eqref{bsde object} and \eqref{bsde variation}, we get
	\[
	\left\{
	\begin{aligned}
	-d\tilde{y}_{2}^{\rho}(t) = & \rho^{-1}[f(t,y_{2}^{\rho}(t),z_{2}^{\rho }(t))-f(t,y_{2}^{\ast}(t),z_{2}^{\ast}(t))\\
	& -\rho f_{y}(t,y_{2}^{\ast}(t),z_{2}^{\ast}(t))\hat{y}_{2}(t)-\rho f_{z}(t,y_{2}^{\ast}(t),z_{2}^{\ast}(t))\hat{z}_{2}(t)]dt-\tilde{z} _{2}^{\rho}(t)dW_{t},\\
	\tilde{y}_{2}^{\rho}(T) = &0.
	\end{aligned}
	\right.
	\]
	Set
	\begin{align*}
	A^{\rho}(t) &  =\int_{0}^{1}f_{y}(t,y_{2}^{\ast}(t)+\lambda\rho(\hat{y}_{2}(t)+\tilde{y}_{2}^{\rho}(t)),z_{2}^{\ast}(t)+\lambda\rho(\hat{z}_{2}(t)+\tilde{z}_{2}^{\rho}(t)))d\lambda,\\
	B^{\rho}(t) &  =\int_{0}^{1}f_{z}(t,y_{2}^{\ast}(t)+\lambda\rho(\hat{y}_{2}(t)+\tilde{y}_{2}^{\rho}(t)),z_{2}^{\ast}(t)+\lambda\rho(\hat{z}_{2}(t)+\tilde{z}_{2}^{\rho}(t)))d\lambda,\\
	C^{\rho}(t) &  =[A^{\rho}(t)-f_{y}(t,y_{2}^{\ast}(t),z_{2}^{\ast}(t))]\hat{y}_{2}(t)+[B^{\rho}(t)-f_{z}(t,y_{2}^{\ast}(t),z_{2}^{\ast}(t))]\hat{z}_{2}(t),
	\end{align*}
	we get that
	\[
	\left\{
	\begin{aligned}
	-d\tilde{y}_{2}^{\rho}(t)  = & [A^{\rho}(t)\tilde{y}_{2}^{\rho}(t)+B^{\rho}(t)\tilde{z}_{2}^{\rho}(t)+C^{\rho}(t)]dt-\tilde{z}_{2}^{\rho}(t)dW_{t},\\
	\tilde{y}_{2}^{\rho}(T) = & 0,
	\end{aligned}
	\right.
	\]
	and accordingly,
	\[
	e^{\int_{0}^{t}A^{\rho}(s)ds}\tilde{y}_{2}^{\rho}(t)=\int_{t}^{T}e^{\int_{0}^{s}A^{\rho}(r)dr}C^{\rho}(s)ds-\int_{t}^{T}e^{\int_{0}^{s}A^{\rho}(r)dr}\tilde{z}_{2}^{\rho}(s)dW_{s}^{\rho},
	\]
	where $W_{t}^{\rho}=W_{t}-\int_{0}^{t}B^{\rho}(s)ds$ is a Brownian motion under the equivalent measure $Q^{\rho}$ defined by
	\[
	\frac{dQ^{\rho}}{dP}=e^{-\frac{1}{2}\int_{0}^{T}(B^{\rho}(t))^{2}dt+\int_{0}^{T}B^{\rho}(t)dW_{t}}.
	\]
	As a result, we get
	\[
	e^{2\int_{0}^{t}A^{\rho}(s)ds}(\tilde{y}_{2}^{\rho}(t))^{2}+E^{Q^{\rho}}\left[  \int_{t}^{T}e^{2\int_{0}^{s}A^{\rho}(r)dr}(\tilde{z}_{2}^{\rho}(s))^{2}ds\bigg|\mathcal{F}_{t}\right]  \leq E^{Q^{\rho}}\left[  \left(\int_{t}^{T}e^{\int_{0}^{s}A^{\rho}(r)dr}C^{\rho}(s)ds\right)  ^{2}\bigg|\mathcal{F}_{t}\right],
	\]
	where $E^{Q^{\rho}}[\cdot]$ is the expectation operator under the measure $Q^{\rho}$.

	It follows from Assumption \ref{ass-qua} that
	\begin{align*}
	&  E^{Q^{\rho}}\left[  \left(  \int_{t}^{T}e^{\int_{0}^{s}A^{\rho}(r)dr}C^{\rho}(s)ds\right)  ^{2}\bigg|\mathcal{F}_{t}\right]  \\
	&  \leq e^{CT}E^{Q^{\rho}}\left[  \left(  \int_{t}^{T}|[A^{\rho}(s)-f_{y}(s,y_{2}^{\ast}(s),z_{2}^{\ast}(s))]\hat{y}_{2}(s)+[B^{\rho}(s)-f_{z}(s,y_{2}^{\ast}(s),z_{2}^{\ast}(s))]\hat{z}_{2}(s)|ds\right)^{2}\bigg|\mathcal{F}_{t}\right]  \\
	&  \leq2e^{CT}E^{Q^{\rho}}\left[  \int_{t}^{T}|A^{\rho}(s)-f_{y}(s,y_{2}^{\ast}(s),z_{2}^{\ast}(s))|^{2}ds\int_{t}^{T}|\hat{y}_{2}(s)|^{2}ds\bigg|\mathcal{F}_{t}\right]  \\
	&  \quad+2e^{CT}E^{Q^{\rho}}\left[  \int_{t}^{T}|B^{\rho}(s)-f_{z}(s,y_{2}^{\ast}(s),z_{2}^{\ast}(s))|^{2}ds\int_{t}^{T}|\hat{z}_{2}(s)|^{2}ds\bigg|\mathcal{F}_{t}\right]  \\
	&  \leq2e^{CT}T\Vert\hat{y}_{2}\Vert_{\mathcal{S}^{\infty}}^{2}\frac{2C^{2}\rho^{2}}{3}E^{Q^{\rho}}\left[  \int_{t}^{T}\left(  |\hat{y}_{2}(s)+\tilde{y}_{2}^{\rho}(s)|^{2}+|\hat{z}_{2}(s)+\tilde{z}_{2}^{\rho}(s)|^{2}\right)  ds\bigg|\mathcal{F}_{t}\right]  \\
	&  \quad+2e^{CT}\frac{2C^{2}\rho^{2}}{3}E^{Q^{\rho}}\left[  \int_{t}^{T}\left(  |\hat{y}_{2}(s)+\tilde{y}_{2}^{\rho}(s)|^{2}+|\hat{z}_{2}(s)+\tilde{z}_{2}^{\rho}(s)|^{2}\right)  ds\int_{t}^{T}|\hat{z}_{2}(s)|^{2}ds\bigg|\mathcal{F}_{t}\right]  \\
	&  \leq2e^{CT}T\Vert\hat{y}_{2}\Vert_{\mathcal{S}^{\infty}}^{2}\frac{4C^{2}\rho^{2}}{3}\left(  T\left(  \Vert\hat{y}_{2}\Vert_{\mathcal{S}^{\infty}}^{2}+\Vert\tilde{y}_{2}^{\rho}\Vert_{\mathcal{S}^{\infty}}^{2}\right)  +\Vert\hat{z}_{2}\Vert_{\text{\textrm{BMO}}(Q^{\rho})}^{2}+\Vert\tilde{z}_{2}^{\rho}\Vert_{\text{\textrm{BMO}}(Q^{\rho})}^{2}\right)  \\
	&  \quad+2e^{CT}\frac{2C^{2}\rho^{2}}{3}\left(  2T\left(  \Vert\hat{y}_{2}\Vert_{\mathcal{S}^{\infty}}^{2}+\Vert\tilde{y}_{2}^{\rho}\Vert
	_{\mathcal{S}^{\infty}}^{2}\right)  \Vert\hat{z}_{2}\Vert_{\text{\textrm{BMO}}(Q^{\rho})}^{2}+E^{Q^{\rho}}\left[  \int_{t}^{T}|\hat{z}_{2}(s)+\tilde{z}_{2}^{\rho}(s)|^{2}ds\int_{t}^{T}|\hat{z}_{2}(s)|^{2}ds\bigg|\mathcal{F}_{t}\right]  \right), 
	\end{align*}
	where $\Vert\cdot\Vert_{\text{\textrm{BMO}}(Q^{\rho})}$ is the BMO-norm with respect to probability $Q^{\rho}$.
	By  H\"{o}lder's inequality,
	\begin{align*}
	&  E^{Q^{\rho}}\left[  \int_{t}^{T}|\hat{z}_{2}(s)+\tilde{z}_{2}^{\rho}(s)|^{2}ds\int_{t}^{T}|\hat{z}_{2}(s)|^{2}ds\bigg|\mathcal{F}_{t}\right]  \\
	&  \leq E^{Q^{\rho}}\left[  \left(  \int_{t}^{T}|\hat{z}_{2}(s)+\tilde{z}_{2}^{\rho}(s)|^{2}ds\right)  ^{2}\bigg|\mathcal{F}_{t}\right]  ^{\frac{1}{2}}E^{Q^{\rho}}\left[  \left( \int_{t}^{T}|\hat{z}_{2}(s)|^{2}ds\right)^{2}\bigg|\mathcal{F}_{t}\right]  ^{\frac{1}{2}}.
	\end{align*}
	Moreover, it follows from \cite{kazamaki} that there exists a constant $K$, which is  independent of $\rho$, $\xi$, and $\xi^{\ast}$, such that
	\[
	E^{Q^{\rho}}\left[  \left(  \int_{t}^{T}|\hat{z}_{2}(s)+\tilde{z}_{2}^{\rho}(s)|^{2}ds\right)  ^{2}\bigg|\mathcal{F}_{t}\right]  ^{\frac{1}{2}}\leq K\Vert\hat{z}_{2}+\tilde{z}_{2}^{\rho}\Vert_{\text{\textrm{BMO}}(Q^{\rho})}^{2}\leq2K\big(\Vert\hat{z}_{2}\Vert_{\text{\textrm{BMO}}(Q^{\rho})}^{2}+\Vert\tilde{z}_{2}^{\rho}\Vert_{\text{\textrm{BMO}}(Q^{\rho})}^{2}\big)
	\]
	and
	\[
	E^{Q^{\rho}}\left[  \left(  \int_{t}^{T}|\hat{z}_{2}(s)|^{2}ds\right)^{2}\bigg|\mathcal{F}_{t}\right]  ^{\frac{1}{2}}\leq K\Vert\hat{z}_{2}\Vert_{\text{\textrm{BMO}}(Q^{\rho})}^{2}.
	\]
	Thus, we obtain
	\begin{align*}
	&  e^{2\int_{0}^{t}A^{\rho}(s)ds}(\tilde{y}_{2}^{\rho}(t))^{2}+E^{Q^{\rho}}\left[  \int_{t}^{T}e^{2\int_{0}^{s}A^{\rho}(r)dr}(\tilde{z}_{2}^{\rho}(s))^{2}ds\bigg|\mathcal{F}_{t}\right]  \\
	&  \leq2e^{CT}T\Vert\hat{y}_{2}\Vert_{\mathcal{S}^{\infty}}^{2}\frac{4C^{2}\rho^{2}}{3}\left[  T\left(  \Vert\hat{y}_{2}\Vert_{\mathcal{S}^{\infty}}^{2}+\Vert\tilde{y}_{2}^{\rho}\Vert_{\mathcal{S}^{\infty}}^{2}\right)  +\Vert\hat{z}_{2}\Vert_{\text{\textrm{BMO}}(Q^{\rho})}^{2}+\Vert\tilde{z}_{2}^{\rho}\Vert_{\text{\textrm{BMO}}(Q^{\rho})}^{2}\right]  \\
	&  \quad+2e^{CT}\frac{4C^{2}\rho^{2}}{3}\Vert\hat{z}_{2}\Vert_{\text{\textrm{BMO}}(Q^{\rho})}^{2}\left[  T\left(  \Vert\hat{y}_{2}\Vert_{\mathcal{S}^{\infty}}^{2}+\Vert\tilde{y}_{2}^{\rho}\Vert_{\mathcal{S}^{\infty}}^{2}\right)  +K^{2}(\Vert\hat{z}_{2}\Vert_{\text{\textrm{BMO}}(Q^{\rho})}^{2}+\Vert\tilde{z}_{2}^{\rho}\Vert_{\text{\textrm{BMO}}(Q^{\rho})}^{2})\right],
	\end{align*}
	which further implies
	\begin{align*}
	&  \Vert\tilde{y}_{2}^{\rho}\Vert_{\mathcal{S}^{\infty}}^{2}+\Vert\tilde{z}_{2}^{\rho}\Vert_{\text{\textrm{BMO}}(Q^{\rho})}^{2}\\
	&  \leq4e^{3CT}T\Vert\hat{y}_{2}\Vert_{\mathcal{S}^{\infty}}^{2}\frac{4C^{2}\rho^{2}}{3}\left[  T\left(  \Vert\hat{y}_{2}\Vert_{\mathcal{S}^{\infty}}^{2}+\Vert\tilde{y}_{2}^{\rho}\Vert_{\mathcal{S}^{\infty}}^{2}\right)  +\Vert\hat{z}_{2}\Vert_{\text{\textrm{BMO}}(Q^{\rho})}^{2}+\Vert\tilde{z}_{2}^{\rho}\Vert_{\text{\textrm{BMO}}(Q^{\rho})}^{2}\right]  \\
	&  \quad+4e^{3CT}\frac{4C^{2}\rho^{2}}{3}\Vert\hat{z}_{2}\Vert_{\text{\textrm{BMO}}(Q^{\rho})}^{2}\left[  T\left(  \Vert\hat{y}_{2}\Vert_{\mathcal{S}^{\infty}}^{2}+\Vert\tilde{y}_{2}^{\rho}\Vert_{\mathcal{S}^{\infty}}^{2}\right)  +K^{2}(\Vert\hat{z}_{2}\Vert_{\text{\textrm{BMO}}(Q^{\rho})}^{2}+\Vert\tilde{z}_{2}^{\rho}\Vert_{\text{\textrm{BMO}}(Q^{\rho})}^{2})\right]  .
	\end{align*}
	Noting  that
	\[
	|B^{\rho}(t)|\leq C(1+|z_{2}^{\ast}(t)|+|z_{2}^{\rho}(t)|)
	\]
	and that there exists a constant $L\geq0$ which only depends on $\xi$, $\xi^{\ast}$ and $C$ such that $\Vert z_{2}^{\ast}\Vert_{\text{\textrm{BMO}}}\leq L$ and $\Vert z_{2}^{\rho}\Vert_{\text{\textrm{BMO}}}\leq L$, it follows from \cite{kazamaki} that there exist two constants $K_{1}>0$ and $K_{2}>0$ which only depend on $L$ and $C$ such that for any $z\in\text{\textrm{BMO}}$, it holds that
	\[
	K_{1}\Vert z\Vert_{\text{\textrm{BMO}}}^{2}\leq\Vert z\Vert_{\text{\textrm{BMO}}(Q^{\rho})}^{2}\leq K_{2}\Vert z\Vert_{\text{\textrm{BMO}}}^{2}.
	\]
	Thus, we have
	\begin{align*}
	&  \Vert\tilde{y}_{2}^{\rho}\Vert_{\mathcal{S}^{\infty}}^{2}+K_{1}\Vert\tilde{z}_{2}^{\rho}\Vert_{\text{\textrm{BMO}}}^{2}\\
	&  \leq4e^{3CT}T\Vert\hat{y}_{2}\Vert_{\mathcal{S}^{\infty}}^{2}\frac{4C^{2}\rho^{2}}{3}\left[  T\left(  \Vert\hat{y}_{2}\Vert_{\mathcal{S}^{\infty}}^{2}+\Vert\tilde{y}_{2}^{\rho}\Vert_{\mathcal{S}^{\infty}}^{2}\right)  +K_{2}\Vert\hat{z}_{2}\Vert_{\text{\textrm{BMO}}}^{2}+K_{2}\Vert\tilde{z}_{2}^{\rho}\Vert_{\text{\textrm{BMO}}}^{2}\right]  \\
	&  \quad+4e^{3CT}\frac{4C^{2}\rho^{2}}{3}K_{2}\Vert\hat{z}_{2}\Vert_{\text{\textrm{BMO}}}^{2}\left[  T\left(  \Vert\hat{y}_{2}\Vert_{\mathcal{S}^{\infty}}^{2}+\Vert\tilde{y}_{2}^{\rho}\Vert_{\mathcal{S}^{\infty}}^{2}\right)  +K_{2}K^{2}(\Vert\hat{z}_{2}\Vert_{\text{\textrm{BMO}}}^{2}+\Vert\tilde{z}_{2}^{\rho}\Vert_{\text{\textrm{BMO}}}^{2})\right]  .
	\end{align*}
	By letting $\rho$ be small enough such that
	\[
	\frac{16C^{2}\rho^{2}T^{2}}{3}e^{3CT}\Vert\hat{y}_{2}\Vert_{\mathcal{S}^{\infty}}^{2}\leq\frac{1}{4},\quad \frac{16C^{2}\rho^{2}K_{2}T}{3}e^{3CT}\Vert\hat{y}_{2}\Vert_{\mathcal{S}^{\infty}}^{2}\leq\frac{K_{1}}{4}
	\]
	and
	\[
	\frac{16C^{2}\rho^{2}K_{2}T}{3}e^{3CT}\Vert\hat{z}_{2}\Vert_{\text{\textrm{BMO}}}^{2}\leq\frac{1}{4},\quad \frac{16C^{2}\rho^{2}K_{2}^{2}K^{2}}{3}e^{3CT}\Vert\hat{z}_{2}\Vert_{\text{\textrm{BMO}}}^{2}\leq\frac{1}{4},
	\]
	we have
	\begin{align*}
	\frac{1}{2}\Vert\tilde{y}_{2}^{\rho}\Vert_{\mathcal{S}^{\infty}}^{2}+\frac{K_{1}}{2}\Vert\tilde{z}_{2}^{\rho}\Vert_{\text{\textrm{BMO}}}^{2} &
	\leq4e^{3CT}T\Vert\hat{y}_{2}\Vert_{\mathcal{S}^{\infty}}^{2}\frac{4C^{2}\rho^{2}}{3}\left(  T\Vert\hat{y}_{2}\Vert_{\mathcal{S}^{\infty}}^{2}+K_{2}\Vert\hat{z}_{2}\Vert_{\text{\textrm{BMO}}}^{2}\right)  \\
	&  \quad+4e^{3CT}\frac{4C^{2}\rho^{2}}{3}K_{2}\Vert\hat{z}_{2}\Vert_{\text{\textrm{BMO}}}^{2}\left(  T\Vert\hat{y}_{2}\Vert_{\mathcal{S}^{\infty}}^{2}+K_{2}K^{2}\Vert\hat{z}_{2}\Vert_{\text{\textrm{BMO}}}^{2}\right)  .
	\end{align*}
	Therefore,  
	\[
	\lim_{\rho\rightarrow0}\Vert\tilde{y}_{2}^{\rho}\Vert_{\mathcal{S}^{\infty}}=0,\text{ and }~\lim_{\rho\rightarrow0}\Vert\tilde{z}_{2}^{\rho}\Vert_{\text{\textrm{BMO}}}=0.
	\]
	We can similarly show that $
	\lim_{\rho\rightarrow0}\Vert\tilde{y}_{1}^{\rho}\Vert_{\mathcal{S}^{\infty}}=0$ and $\lim_{\rho\rightarrow0}\Vert\tilde{z}_{1}^{\rho}\Vert_{\text{\textrm{BMO}}}=0$.
\end{proof}

\begin{lemma} \label{lemma-h1-h2}
	Suppose that $f$ and $g$ satisfy Assumptions~\ref{ass-qua} and \ref{ass-ini} and $h$ obeys Assumption \ref{ass-h}. Then, there exist $h_{1}\in\mathbb{R}$, $h_{2}\in\mathbb{R}$ with $h_{1}\geq0$ and $|h_{1}|^{2}+|h_{2}|^{2}=1$, such that the following variational inequality holds
	\[
	h_{1}\hat{y}_{1}(0)+h_{2}\hat{y}_{2}(0)\geq0.
	\]
	
\end{lemma}

\begin{proof}
	We will first show the proof  for the case where $\mathcal{E}_{g}[\xi^{\ast}]=\pi_{0}$.
	Define
	\[
	F_{\epsilon}(\xi)=\{(\mathcal{E}_{g}[\xi]-\pi_{0})^{2}+[(\mathcal{E}_{f}[h(\xi)+\alpha\mathcal{E}_{g}[\xi]]-\mathcal{E}_{f}[h(\xi^{\ast})+\alpha\mathcal{E}_{g}[\xi^{\ast}]])+\epsilon)^{+}]^{2}\}^{\frac{1}{2}}
	\]
	for positive constant $\epsilon$.

	It is easy to verify that $F_{\epsilon}(\cdot)$ is a continuous functional on $U$ with respect to $\|\cdot\|_{\infty}$ and satisfies
	\[
	\left\{
	\begin{aligned} &F_{\epsilon}(\xi)>0,\ \forall\xi\in U,\\ &F_{\epsilon}(\xi^{\ast})=\epsilon\leq\inf\limits_{\xi\in U}F_{\epsilon}(\xi)+\epsilon. \end{aligned} \right.
	\]
	Thus, by Ekeland's variational principle, there exists a $\xi^{\epsilon}\in U$ such that
	\[
	F_{\epsilon}(\xi^{\epsilon})\leq F_{\epsilon}(\xi^{\ast})=\epsilon,\ \ \|\xi^{\ast}-\xi^{\epsilon}\|_{\infty}\leq\sqrt{\epsilon},
	\]
	and
	\begin{equation}
	-\sqrt{\epsilon}\|\xi-\xi^{\epsilon}\|_{\infty}\leq F_{\epsilon}(\xi)-F_{\epsilon}(\xi^{\epsilon}),\ \ \ \forall\xi\in U \label{ekeland}
	\end{equation}
	which means $\xi^{\epsilon}$ is a minimizer (over $U$) of the cost functional $F_{\epsilon}(\xi)+\sqrt{\epsilon}\|\xi-\xi^{\epsilon}\|_{\infty}$.
	
	Set $\hat{\xi}=\xi-\xi^{\ast},\hat{\xi}^{\epsilon}=\xi-\xi^{\epsilon},\xi_{\rho}^{\epsilon}=\xi^{\epsilon}+\rho\hat{\xi}^{\epsilon}$, and consider the following variational equation:
	\[
	\left\{
	\begin{aligned}
	-d\hat{y}_{1}^{\epsilon}(t)&=[f_{y}(t,y_{1}^{\epsilon}(t),z_{1}^{\epsilon }(t))\hat{y}_{1}^{\epsilon}(t)+f_{z}(t,y_{1}^{\epsilon}(t),z_{1}^{\epsilon}(t))\hat{z}_{1}^{\epsilon}(t)]dt-\hat{z}_{1}^{\epsilon}(t)dW_{t},\\
	\hat{y}_{1}^{\epsilon}(T)&=h_{x}(\xi^{\epsilon})\hat{\xi}^{\epsilon}+\alpha\hat{y}_{2}^{\epsilon}(0),\\
	-d\hat{y}_{2}^{\epsilon}(t)&=[g_{y}(t,y_{2}^{\epsilon}(t),z_{2}^{\epsilon }(t))\hat{y}_{2}^{\epsilon}(t)+g_{z}(t,y_{2}^{\epsilon}(t),z_{2}^{\epsilon }(t))\hat{z}_{2}^{\epsilon}(t)]dt-\hat{z}_{2}^{\epsilon}(t)dW_{t},\\
	\hat{y}_{2}^{\epsilon}(T)&=\hat{\xi}^{\epsilon}.
	\end{aligned}
	\right.
	\]
	where $ (y_{1}^{\epsilon}, z_{1}^{\epsilon}, y_{2}^{\epsilon}, z_{2}^{\epsilon})\in\mathcal{S}^{\infty}\times\text{\textrm{BMO}}\times\mathcal{S}^{\infty}\times\text{\textrm{BMO}}$ is the unique solution to \eqref{bsde object} for $\xi=\xi^{\epsilon}$. We apply  Lemma \ref{lemma estimation} to obtain
	\begin{align*}
	& \lim_{\rho\rightarrow0}\left|\frac{1}{\rho}(\mathcal{E}_{g}[\xi_{\rho}^{\epsilon}]-\mathcal{E}_{g}[\xi^{\epsilon}])-\hat{y}_{2}^{\epsilon}(0)\right|=0,\\
	& \lim_{\rho\rightarrow0}\left|\frac{1}{\rho}\{\mathcal{E}_{f}[h(\xi_{\rho}^{\epsilon})+\alpha\mathcal{E}_{g}[\xi_{\rho}^{\epsilon}]]-\mathcal{E}_{f}[h(\xi^{\epsilon})+\alpha\mathcal{E}_{g}[\xi^{\epsilon}]]\}-\hat{y}_{1}^{\epsilon}(0)\right|=0.
	\end{align*}
	The above two equations imply
	\begin{eqnarray*}
		|\mathcal{E}_{g}[\xi_{\rho}^{\epsilon}]-\pi_{0}|^{2}-|\mathcal{E}_{g}[\xi^{\epsilon}]-\pi_{0}|^{2}=2\rho\hat{y}_{2}^{\epsilon}(0)(\mathcal{E}_{g}[\xi^{\epsilon}]-\pi_{0})+o(\rho),
	\end{eqnarray*}
	and
	\begin{eqnarray*}
		&& |\mathcal{E}_{f}[h(\xi_{\rho}^{\epsilon})+\alpha\mathcal{E}_{g}[\xi_{\rho}^{\epsilon}]]-\mathcal{E}_{f}[h(\xi^{\ast})+\alpha\mathcal{E}_{g}[\xi^{\ast}]]+\epsilon|^{2}-|\mathcal{E}_{f}[h(\xi^{\epsilon})+\alpha\mathcal{E}_{g}[\xi^{\epsilon}]]-\mathcal{E}_{f}[h(\xi^{\ast})+\alpha\mathcal{E}_{g}[\xi^{\ast}]]+\epsilon|^{2}\\
		&& \ \, =2\rho\hat{y}_{1}^{\epsilon}(0)(\mathcal{E}_{f}[h(\xi^{\epsilon})+\alpha\mathcal{E}_{g}[\xi^{\epsilon}]]-\mathcal{E}_{f}[h(\xi^{\ast})+\alpha\mathcal{E}_{g}[\xi^{\ast}]]+o(\rho),
	\end{eqnarray*}
	where $o(\rho)\to 0$ as $\rho\to 0$.
	
	We consider two exclusive cases in the subsequent analysis:
	\newline\emph{Case} 1: For an arbitrary small $\rho>0$, $\mathcal{E}_{f}[h(\xi_{\rho}^{\epsilon})+\alpha\mathcal{E}_{g}[\xi_{\rho}^{\epsilon}]]-\mathcal{E}_{f}[h(\xi^{\ast})+\alpha\mathcal{E}_{g}[\xi^{\ast}]]+\epsilon>0$. In this case,
	\begin{align*}
	\begin{split}
	& \lim\limits_{\rho\rightarrow0}\frac{F_{\epsilon}(\xi_{\rho}^{\epsilon})-F_{\epsilon}(\xi^{\epsilon})}{\rho}\\
	& \ \ =\lim\limits_{\rho\rightarrow0}\frac{1}{F_{\epsilon}(\xi_{\rho}^{\epsilon})+F_{\epsilon}(\xi^{\epsilon})}\frac{|F_{\epsilon}(\xi_{\rho}^{\epsilon})|^{2}-|F_{\epsilon}(\xi^{\epsilon})|^{2}}{\rho}\\
	& \ \ =\frac{1}{F_{\epsilon}(\xi^{\epsilon})}\{\hat{y}_{2}^{\epsilon}(0)(\mathcal{E}_{g}[\xi^{\epsilon}]-\pi_{0})+\hat{y}_{1}^{\epsilon}(0)(\mathcal{E}_{f}[h(\xi^{\epsilon})+\alpha\mathcal{E}_{g}[\xi^{\epsilon}]]-\mathcal{E}_{f}[h(\xi^{\ast})+\alpha\mathcal{E}_{g}[\xi^{\ast}]]+\epsilon)\}.
	\end{split}
	\end{align*}
	Therefore, we set
	\[
	h_{1}^{\epsilon}=\frac{\mathcal{E}_{f}[h(\xi^{\epsilon})+\alpha\mathcal{E}_{g}[\xi^{\epsilon}]]-\mathcal{E}_{f}[h(\xi^{\ast})+\alpha\mathcal{E}_{g}[\xi^{\ast}]]+\epsilon}{F_{\epsilon}(\xi^{\epsilon})}\geq0,\ h_{2}^{\epsilon}=\frac{\mathcal{E}_{g}[\xi^{\epsilon}]-\pi_{0}}{F_{\epsilon}(\xi^{\epsilon})},
	\]
	and use  (\ref{ekeland}) to get
	\begin{equation}
	\label{variation epsilon}
	h_{1}^{\epsilon}\hat{y}_{1}^{\epsilon}(0)+h_{2}^{\epsilon}\hat{y} _{2}^{\epsilon}(0)\geq-\sqrt{\epsilon}\|\xi_{\rho}^{\epsilon}-\xi^{\epsilon}\|_{\infty}.
	\end{equation}
	\emph{Case} 2: There exists a sequence $\{\rho_{n}\}\downarrow0$\ such that $\mathcal{E}_{f}[h(\xi_{\rho}^{\epsilon})+\alpha\mathcal{E}_{g}[\xi_{\rho}^{\epsilon}]]-\mathcal{E}_{f}[h(\xi^{\ast})+\alpha\mathcal{E}_{g}[\xi^{\ast}]]+\epsilon\leq0$. In this case, $F_{\epsilon}(\xi_{\rho_{n}}^{\epsilon})=|\mathcal{E}_{g}[\xi_{\rho_{n}}^{\epsilon}]-\pi_{0}|$, and hence,
	\[
	\lim\limits_{\rho_{n}\rightarrow0}\frac{F_{\epsilon}(\xi_{\rho_{n}}^{\epsilon})-F_{\epsilon}(\xi^{\epsilon})}{\rho_{n}}=\frac{1}{F_{\epsilon}(\xi^{\epsilon})}\hat{y}_{2}^{\epsilon}(0)(\mathcal{E}_{g}[\xi^{\epsilon}]-\pi_{0}).
	\]
	Therefore, we set  $h_{2}^{\epsilon}=\frac{\mathcal{E}_{g}[\xi^{\epsilon}]-\pi_{0}}{F_{\epsilon}(\xi^{\epsilon})}$, and $\ h_{1}^{\epsilon}=0 $ to get  (\ref{variation epsilon}).
	
	The above analysis ensures that, for a given $\epsilon$, we can find $h_1^\epsilon$ and $h_2^\epsilon$ such that
	$h_{1}^{\epsilon}\geq0$, $|h_{1}^{\epsilon}|^{2}+|h_{2}^{\epsilon}|^{2}=1$, and \eqref{variation epsilon} simultaneously hold. Thus, there is a subsequence, which we still denote by $(h_{1}^{\epsilon},h_{2}^{\epsilon})$, such that
	\[
	\lim\limits_{\epsilon\rightarrow0}(h_{1}^{\epsilon},h_{2}^{\epsilon})=(h_{1},h_{2}),
	\]
	for some $(h_{1},h_{2})\in\mathbb{R}^{2}$ with
	\[
	h_{1}\geq0,\ \, \mbox{ and}\ \, |h_{1}|^{2}+|h_{2}|^{2}=1.
	\]
	Furthermore, following the same line as we used in the proof of Lemma
	\ref{lemma estimation}, we can prove
	\[
	\lim_{\epsilon\rightarrow0}\|\hat{y}_{1}^{\epsilon}-\hat{y}_{1}\|_{\mathcal{S}^{\infty}}=0,~\text{and}~\lim_{\epsilon\rightarrow0}\|\hat{y}_{2}^{\epsilon}-\hat{y}_{2}\|_{\mathcal{S}^{\infty}}=0.
	\]
	This completes the proof for the case with $\mathcal{E}_{g}[\xi^{\ast}]=\pi_{0}$.

		If $\mathcal{E}_{g}[\xi^{\ast}]<\pi_{0}$,  by redefining $F_{\epsilon}(\xi)$ as
		\[
		F_{\epsilon}(\xi)=\{[(\mathcal{E}_{f}[h(\xi)+\alpha\mathcal{E}_{g}[\xi]]-\mathcal{E}_{f}[h(\xi^{\ast})+\alpha\mathcal{E}_{g}[\xi^{\ast}]])+\epsilon)^{+}]^{2}\}^{\frac{1}{2}},
		\]
		 the proof follows similarly to the case $ \mathcal{E}_{g}[\xi^{\ast}]=\pi_{0}$. 
\end{proof}

Now we are ready to present the proof of Theorem \ref{theorem}.

\begin{proof}
	[Proof of Theorem \ref{theorem}]  The proof will follow from the variational inequality $h_{1}\hat{y}_{1}(0)+h_{2}\hat{y}_{2}(0)\geq0$
	in Lemma ~\ref{lemma-h1-h2}.  By virtue of (\ref{adjoint}) and  (\ref{bsde variation}), it is easy to verify that both $\{\hat{y}_{1}(t)n(t), t\in [0, T]\}$ and $\{\hat{y}_{2}(t)m(t), t\in [0, T]\}$ are martingales, whereby
	\begin{equation}\label{y2-temp}
	\hat{y}_{2}(0)=\hat{y}_{2}(0)m(0)=E[m(T)\hat{y}_{2}(T)]=E[m(T)\cdot(\xi-\xi^{\ast})],
	\end{equation}
	and
	\[
	E[h_{1}\hat{y}_{1}(T)n(T)+h_{2}\hat{y}_{2}(T)m(T)]=h_{1}\hat{y}_{1}(0)+h_{2}\hat{y}_{2}(0)\geq0.
	\]
	Recallinig  $\hat y_1(T)$ and $\hat y_2(T)$ given in (\ref{bsde variation}) and substituting them  into the above equation yields
	\begin{equation}
	\label{inequality}
	E[(h_{1}{ h_x}(\xi^{\ast})n(T)+h_{2}m(T))\cdot(\xi-\xi^{\ast})]+\alpha h_{1}E[n(T)]\hat{y}_{2}(0)\geq0.
	\end{equation}
	Moreover, plugging \eqref{y2-temp} into  (\ref{inequality}) yields
	\begin{equation}
	\label{inequality 2}
	E[(h_{2}m(T)+h_{1}h_{x}(\xi^{\ast})n(T)+\alpha h_{1}m(T)E[n(T)])\cdot(\xi-\xi^{\ast})]\geq0.
	\end{equation}
	Since the above inequality holds for arbitrary $\xi\in U$, we claim that the optimal solution $\xi^*$ must take the form:
	\[
	\xi^{\ast}=
	\left\{
	\begin{aligned}
	&X, &\text{on}\ \{h_{2}m(T)+h_{1}h_{x}(\xi^{\ast})n(T)+\alpha h_{1}m(T)E[n(T)]>0\};\\
	&b,&\text{on}\ \{h_{2}m(T)+h_{1}h_{x}(\xi^{\ast})n(T)+\alpha h_{1}m(T)E[n(T)]=0\};\\
	&Y, &\text{on}\ \{h_{2}m(T)+h_{1}h_{x}(\xi^{\ast})n(T)+\alpha h_{1}m(T)E[n(T)]<0\},
	\end{aligned}
	\right.
	\]
	where $b\in L^{\infty}(\mathcal{F}_{T})$ satisfying $X\leq b\leq Y$. If $ \xi^*$ is not given as claimed above, for example, $\xi^{\ast}>X, $ \ on $\{h_{2}m(T)+h_{1}h_{\xi}(\xi^{\ast})n(T)+\alpha h_{1}m(T)E[n(T)]>0\}$ and $P\{h_{2}m(T)+h_{1}h_{\xi}(\xi^{\ast})n(T)+\alpha h_{1}m(T)E[n(T)]>0\}>0$, then we define
	\[
	\xi=
	\left\{
	\begin{aligned}
	&{ \frac{1}{2}(\xi^{\ast}+X),} & \  \text{on}\ \{h_{2}m(T)+h_{1}h_{x}(\xi^{\ast})n(T)+\alpha h_{1}m(T)E[n(T)]>0\};\\
	&\xi^{\ast},&\ \text{on}\ \{h_{2}m(T)+h_{1}h_{x}(\xi^{\ast})n(T)+\alpha h_{1} m(T)E[n(T)]\leq0\},
	\end{aligned}
	\right.
	\]
	to obtain
	\begin{align*}
	&  E[(h_{2}m(T)+h_{1}h_{x}(\xi^{\ast})n(T)+\alpha h_{1}m(T)E[n(T)])\cdot(\xi-\xi^{\ast})]\\
	&  = \frac{1}{2}\int_A(h_{2}m(T)+h_{1}h_{x}(\xi^{\ast})n(T)+\alpha h_{1}m(T)E[n(T)])\cdot (X-\xi^{\ast})dP<0,
	\end{align*}
	where $A:=\{h_{2}m(T)+h_{1}h_{x}(\xi^{\ast})n(T)+\alpha h_{1}m(T)E[n(T)]>0\}$.
	But this is a contradiction to (\ref{inequality 2}). Other cases can be addressed in the same way. This completes the proof.
\end{proof}

We now provide conditions under which there exists an optimal control.

\begin{proposition}
	\label{prop-ex} Suppose that $f$ and $g$ satisfy Assumptions~\ref{ass-qua} and \ref{ass-ini}, $h$ obeys Assumption~\ref{ass-h}. Further assume that $f$, $g$ are convex in  $(y,z)$ and $h$ is convex in $x$. Then, there exists an optimal solution to problem \eqref{g-object}.
\end{proposition}

\begin{proof}
	By a standard argument, it is easy to check that the functionals
	\[
	\xi\rightarrow\mathcal{E}_{g}[\xi]~\text{and}~\xi\rightarrow\mathcal{E}_{f}[h(\xi)+\alpha\mathcal{E}_{g}[\xi]]
	\]
	are convex.  Therefore, it follows from \cite[Proposition 2.3]{briand} that $\xi\rightarrow\mathcal{E}_{g}[\xi]$ and $\xi\rightarrow\mathcal{E}_{f}[h(\xi)+\alpha\mathcal{E}_{g}[\xi]]$ are strongly continuous,  and thus weakly lower-semicontinuous.
	Furthermore, $U$ is weakly compact since it is convex, closed and bounded.  Therefore,  the minimum in problem~\eqref{g-object} is attained.
\end{proof}

In general the optimal solution $\xi^{\ast}$ may not satisfy $\mathcal{E}_{g}[\xi^{\ast}]=\pi_{0}$. In the following, we provide sufficient conditions under which the constraint is binding upon the optimal solution $\xi^{\ast}$.

\begin{proposition}
	Suppose that $f$ and $g$ satisfy Assumptions~\ref{ass-qua} and \ref{ass-ini}, $h$ obeys Assumption~\ref{ass-h}, $\alpha=0$ and $h$ is strictly decreasing.  If $\xi^{\ast}$ is an optimal solution to  problem \eqref{g-object}, then $\xi^{\ast}$ must satisfy $\mathcal{E}_{g}[\xi^{\ast}]=\pi_{0}$.
\end{proposition}

\begin{proof} 
	Suppose by way of contradiction that $\xi^{\ast}$ is optimal but $\mathcal{E}_{g}[\xi^{\ast}]<\pi_{0}$. Then,  it follows from \cite[Proposition 2.3]{briand} that one can construct a random variable $\hat{\xi}$ such that $\hat{\xi}\geq\xi^{\ast}$, $\mathbb{P}(\hat{\xi}> \xi^{\ast})>0$ and $\mathcal{E}_{g}[\hat{\xi}]\leq \pi_{0}$.  Moreover, we have $h(\xi^{\ast})\geq h(\hat{\xi})$, $\mathbb{P}(h(\xi^{\ast})> h(\hat{\xi}))>0$ since $h$ is strictly deceasing. Therefore, it follows from the strict comparison theorem for quadratic BSDEs as stated in \cite[Theorem 3.2]{ma} that $\mathcal{E}_{f}[h(\xi^{\ast})]>\mathcal{E}_{f}[h(\hat{\xi})]$, which contradicts the optimality of $\xi^{\ast}$.
\end{proof}

Relying on the classical Lagrange approach and combining with Lemma \ref{lemma estimation}, the following result follows directly from a similar technique as in \cite{Ji-Zhou-PT}.

\begin{proposition}
	\label{prop-lagran}
	Suppose that $f$ and $g$ satisfy Assumptions \ref{ass-qua} and \ref{ass-ini}, $h$ obeys Assumption \ref{ass-h}, $\alpha=0$, $f$, $g$ and $h$ are convex, and $h$ is strictly decreasing. Then, there exists an optimal solution  $\xi^{\ast}$ to problem \eqref{g-object} satisfying $\mathcal{E}_{g}[\xi^{\ast}]= \pi_{0}$ and
	\begin{align}
	\label{sou-lagran}
	\xi^{\ast}=YI_{\{v m(T)+h_{x}(\xi^{\ast})n(T)<0\}}+XI_{\{vm(T)+h_{x}(\xi^{\ast})n(T)>0\}}+bI_{\{v m(T)+h_{x}(\xi^{\ast})n(T)=0\}}
	\end{align}
	for some constant $v>0$ and $b\in L^{\infty}(\mathcal{F}_{T})$ satisfying $X\leq b\leq Y$, where  $m$ and $n$ are given in \eqref{adjoint}. Conversely, if there exist constant $v>0$ and $b\in L^{\infty}(\mathcal{F}_{T})$ satisfying $X\leq b\leq Y$ such that $\xi^{\ast}$ given by \eqref{sou-lagran} satisfies $\mathcal{E}_{g}[\xi^{\ast}]= \pi_{0}$, then $\xi^{\ast}$ is an optimal solution to problem \eqref{g-object}.
\end{proposition}

\section{Applications} \label{sec:app}

\subsection{ Partial hedging with ambiguity}

Partial hedging refers to the situation in which  a bounded contingent claim $X$ is hedged only partially at a hedging cost $\pi_0$ that is smaller than the cost  of  perfectly replicating or superhedging the claim. In the context of a complete market model with pricing measure $Q$, the partial hedging problem can be formulated as follows,
\begin{equation}\label{partial hedging problem}
\text{minimize $\rho(X-\xi)$ over random variables $\xi$ with $0\le\xi\le X$ and $E^Q[\xi]\le\pi_0$,}
\end{equation}
where $\rho$ is a convex risk measure and $\pi_0\in(0,E^Q[X])$ is the cost constraint; see Section 3 in \cite{SchiedPuebla} for the link between the static problem \eqref{partial hedging problem} and the actual dynamic hedging problem as well as for explicit solutions in some special cases, including the case in which $\rho$ is Expected Shortfall (sometimes also called Average Value at Risk).
Starting from a number of seminal publications such as \cite{CvitanicKaratzas,Kulldorff,follmer1,Sekine}, the  problem \eqref{partial hedging problem} has received significant attention in the past two decades; see, e.g., \cite{rudloff,melnikov, cong1,cong2}.  We refer to  \cite{EmbrechtsSchiedWang} for some recent advances in the case where the risk measure $\rho$ is equal to Value-at-Risk,  a utility-based shortfall risk measure, or a divergence risk measure.  With the exception of the entropic risk measure, however, all these risk measures do not satisfy our present condition of dynamic consistency \cite{KupperSchachermayer}.

In the following,  we consider  a partial hedging problem under the general  $g$-expectation. That is,
\begin{equation}
\left\{
\begin{aligned}
\min\limits_{\xi\in U} &\  \mathcal{E}_f[ X-\xi]\\
s.t. & \ \mathcal{E}_g[\xi]\leq\pi_{0},
\end{aligned}
\right.
\label{ambi-gener-qua}
\end{equation}
where  the quadratic generator $ f(t, y, z)$ satisfies  Assumptions \ref{ass-qua} and \ref{ass-ini} and $ g(t, y, z)=\mu_{t} z$ for a bounded adapted process $\mu_t$.
Define $Q$  by
\begin{equation*}
\frac{dQ}{dP}:=e^{-\frac{1}{2}\int_{0}^{T}|\mu_{t}|^{2}dt-\int_{0}^{T}\mu_{t}dW_{t}},
\end{equation*}
then $\mathcal{E}_g[\xi]=E^{Q}[\xi]$.
In addition,  we assume that $f(t,y,z)$ is convex with respect to $(y,z)$.  Hence $\mathcal{E}_f[\cdot]$ is a strongly time-consistent convex risk measure (see \cite{gianin,DelbaenPengRosazza}).  The notion of strong time consistency means that
$$\mathcal{E}_f[\mathcal{E}_f[\xi|\mathcal{F}_{t}]|\mathcal{F}_s]= \mathcal{E}_f[\xi|\mathcal{F}_{s}], \ 0\leq s\leq t \leq T. $$
We refer  Peng \cite{peng2004} and  Ma and Yao  \cite{ma} for a detailed discussion in the context of $g$-expectations. For general dynamic risk measures,  strong time consistency has the following interpretation. If the risk of a financial position $\xi$ is first assessed at time $t$ and the risk of that assessment is re-assessed at an earlier time $s\le t$, then the result is identical to the result when assessing the risk of $\xi$ directly at time $s$. From a mathematical point, strong time consistency is similar to the Bellman principle in stochastic optimal control. It thus facilitates the application of control techniques to solve optimization problems.
Applying  Proposition \ref{prop-lagran}, we have  the following characterization results for the shape of the optimal partial hedging strategy.

\begin{proposition}\label{prop 3.1}
	There exists an optimal solution $\xi^{\ast}$ to problem \eqref{ambi-gener-qua}. Moreover there exists a constant $v>0$ and an $\mathcal{F}_{T}$-measurable random variable $b \in[0,X]$ such that
	\begin{align}
	\label{sou-lagran-ambi}
	\xi^{\ast}=XI_{\{v m(T)-n(T)<0\}}+bI_{\{vm(T)-n(T)=0\}}
	\end{align}
	and $E^{Q}[\xi^{\ast}]= \pi_{0}$, where $m(\cdot)$ and $n(\cdot)$ are solutions to the following stochastic differential equations
	\[
	\left\{
	\begin{aligned}
	dn(t)&=f_{y}(t,y_{1}^{\ast}(t),z_{1}^{\ast}(t))n(t)dt+f_{z}(t,y_{1}^{\ast }(t),z_{1}^{\ast}(t))n(t)dW_{t},\\
	dm(t)&=\mu_t m(t)dW_{t},\\
	n(0)&=1,\ m(0)=1.
	\end{aligned}
	\right.
	\]
	and $z_1^*$ is the second component of the solution to the BSDE 
	\begin{equation}
	\left\{
	\begin{aligned}
	dy_{1}^*(t)&=-f(t,y_1^*(t),z_1^*(t))dt+z_{1}^*(t)dW_{t},\\ y_{1}^*(T)&=X-\xi^*.
	\end{aligned}
	\right.
	\end{equation}
	Conversely, if $\xi^{\ast} \in U$ has the form (\ref{sou-lagran-ambi}) with $E^{Q}[\xi^{\ast}]= \pi_{0}$, then $\xi^{\ast}$ is optimal for problem (\ref{ambi-gener-qua}).
\end{proposition}

\subsection{ Fundraising under ambiguity}

Consider an agent who wants to raise capital $\alpha\geq 0$ today by issuing a contingent claim $\xi$ with a fixed maturity and suppose that the liability of the agent at maturity needs to be bounded by some constant $K$. 
Using the notation of  the preceding  section, we formulate the optimal fundraising problem as follows:
\begin{equation}
\label{liability}
\left\{
\begin{aligned}
\inf\limits_{\xi\in U}& \mathcal{E}_f[ -u(-\xi)]\\
s.t. & \ E^{Q}[\xi]\geq\alpha,
\end{aligned}
\right.
\end{equation}
where $U=\{\xi\in L^{\infty}(\mathcal{F}_{T}):0\leq\xi\leq K\}$, the quadratic generator $ f(t, y, z)$ satisfies  Assumptions \ref{ass-qua} and \ref{ass-ini},  and $u$ is a concave utility function such that $-u$ satisfies Assumption \ref{ass-h}.

The robust  fundraising problem has been approached in a number of settings and with several techniques; see, e.g., \cite{CvitanicKaratzas,FoellmerSchiedWeber,SchiedPuebla} and the references therein. Our discussion here differs from those previous  approaches in at least  two aspects.
First,   we focus on the general $g$-expectations as our cost criterion.
Second, we will characterize the solution to the problem using BSDEs.

We further assume that $\mu_{t}$ is bounded and let $g(t,z)=\mu_{t} z$.  Then, problem~\eqref{liability} is equivalent to
\[
\left\{
\begin{aligned}
\inf\limits_{\tilde{\xi}\in \tilde{U}} &\  \mathcal{E}_f[ -u(\tilde{\xi})]\\
s.t. & \ \mathcal{E}_g[\tilde{\xi}]\leq-\alpha
\end{aligned}
\right.
\]
in the sense that $\xi^*=-\tilde{\xi}^*$ solves problem~\eqref{liability} if and only if $\tilde{\xi}^*$ is a solution to the above problem, \ where $\tilde{U}=\{\tilde{\xi}\in L^{\infty}(\mathcal{F}_{T}):-K\leq\tilde{\xi}\leq0\}$. Applying Proposition \ref{prop-lagran} to the above problem and translating the result for problem~\eqref{liability}, we obtain the following proposition.

\begin{proposition}
	If $\xi^{\ast}$ is an optimal solution to problem \eqref{liability}, then there exist a constant $v>0$ and an $\mathcal{F}_{T}$-measurable random variable $b \in[0,K]$ such that
	\begin{align}
	\label{sou-lagran-amb}
	\xi^{\ast}=KI_{\{v m(T)-u_{x}(-\xi^{\ast})n(T)>0\}}+bI_{\{v m(T)-u_{x}(-\xi^{\ast})n(T)=0\}}
	\end{align}
	and $E^{Q}[\xi^{\ast}]= \alpha$, where $m(\cdot)$, $n(\cdot)$ are solutions
	to
	\[
	\left\{
	\begin{aligned}
	dn(t)&=f_{y}(t,y_{1}^{\ast}(t),z_{1}^{\ast}(t))n(t)dt+f_{z}(t,y_{1}^{\ast }(t),z_{1}^{\ast}(t))n(t)dW_{t},,\\
	dm(t)&=\mu_t m(t)dW_{t},\\ n(0)&=1,\ \ m(0)=1,
	\end{aligned}
	\right.
	\]
	and $z_1^*$ is the second component of the solution to the BSDE 
	\begin{equation}
	\left\{
	\begin{aligned}
	dy_{1}^*(t)&=-f(t,y_1^*(t),z_1^*(t))dt+z_{1}^*(t)dW_{t},\\ y_{1}^*(T)&=-u(-\xi^*).
	\end{aligned}
	\right.
	\end{equation}
	Conversely, if $\xi^{\ast} \in U$ has the form (\ref{sou-lagran-amb}) with $E^{Q}[\xi^{\ast}]= \alpha$, then $\xi^{\ast}$ is optimal for problem (\ref{liability}).
\end{proposition}

\subsection{Neyman--Pearson lemma for quadratic $g$-probabilities}

For a generator  $g$ satisfying Assumptions \ref{ass-qua} and \ref{ass-ini}, define  the $g$-probability as $P_{g}(A)=\mathcal{E}
_{g}[I_{A}]$, for $A\in\mathcal{F}_{T}$.
In this subsection,  we will consider statistical tests for
$g$-probabilities.

Suppose $g$ and $f$ satisfy Assumptions \ref{ass-qua} and \ref{ass-ini}. A test
for a simple null hypothesis $H_{0}: \bar{g}=g$ versus a simple alternative
hypothesis $H_{1}: \bar{g}=f$ is an $\mathcal{F}_{T}$-measurable random variable
$\xi: \Omega\rightarrow\{0,1\}$, which rejects $H_{0}$ on the event $\{\xi=1
\}$. For $0 <\pi_{0} <1$, the test is said to have a significance level
$\pi_{0}$ if $P_{g}(\xi=1)\leq\pi_{0}$. For an outcome of the test, i.e., a
sample $\omega$ from the sample space $\Omega$, the hypothesis  $g$ is
rejected (resp. accepted) if $\omega\in\{\xi=1\}$ (resp. $\omega\in\{\xi=0\}$). Thus, $P_{g}(\xi=1)$ is the $g$-probability of Type I error ( i.e., the probability of
rejecting $H_{0}$ when it is true), whereas $P_{f}(\xi=0)$ is the $g$-probability
of Type II error ( i.e., the probability of accepting $H_{0}$ when it is false). 
A conventional way to approach the hypothesis test is to find a test that minimizes the $g$-probability of a Type II
error among all tests that keep the $g$-probability of a Type I error below a
given acceptable significance level $\pi_{0}\in(0,1)$. The optimal simple test
may not exist  generally, and thus, one needs to consider a randomized test  represented by a  random variable $\xi:\Omega\rightarrow[0,1]$, which means that if the outcome
$\omega\in\Omega$ is observed, then the hypothesis $H_{0}$ is rejected with
probability $\xi(\omega)$. 

In this subsection, we set $f(t,y,z)=\gamma z^{2}$ and $g(t,y,z)=\eta z^{2}$ for $\gamma>0$ and $\eta>0$, and consider the following randomized test problem:
\begin{equation}
\label{g-object-qua1}
\left\{
\begin{aligned}
\inf\limits_{\xi\in U} & \ \mathcal{E}_{f}[1-\xi]\\
s.t. & \ \mathcal{E}_{g}[\xi]\leq\pi_{0},
\end{aligned}
\right.
\end{equation}
where $U=\{\xi\in L^{\infty}(\mathcal{F}_{T}): 0\leq\xi\leq1\}$ and $0<\pi_{0}<1$. If both $g$-expectations reduce to linear operators (i.e., $f(z)=\gamma z$ and $g(z)=\eta z$), the solution to  the above problem is given by the classical Neyman--Pearson lemma, and for this reason it is called a Neyman--Pearson problem under $g$-expectations. When the generators $f$ and $g$ are uniformly Lipschitz continuous with respect to $(y,z)$, the Neyman--Pearson lemma for the above problem has been established by Ji and Zhou \cite{Ji-Zhou-PT}. Here, we focus on generators that are quadratic functions with respect to $z$.

Using a standard exponential transformation, we derive
\[
\mathcal{E}_{f}[1-\xi]=\frac{1}{2\gamma}\ln E[e^{2\gamma(1-\xi)}]\quad \mbox{and}\quad\mathcal{E}_{g}[\xi]=\frac{1}{2\eta}\ln E[e^{2\eta\xi}].
\]
Thus, problem \eqref{g-object-qua1} is equivalent to
\begin{equation}
\left\{
\begin{aligned}
\inf\limits_{\xi\in U} & \ E[e^{2\gamma(1-\xi)}]\\
s.t. & \  E[e^{2\eta\xi}] \leq \tilde{\pi}_{0},
\end{aligned}
\right.
\end{equation}
where $\tilde{\pi}_{0}=e^{2\eta\pi_{0}}$.  Applying Proposition~\ref{prop-lagran}, we can derive an explicit solution for the randomized test problem as given in proposition below.

\begin{proposition} A randomized test 
	$\xi^{\ast}$ is optimal if and only if there exist a positive constant $v\in[\frac{\gamma}{\eta}e^{-2\eta},\frac{\gamma}{\eta}e^{2\gamma}]$ and an $\mathcal{F}_{T}$-measurable set $A$ such that
	\begin{align*}
	\xi^{\ast}=\frac{2\gamma-\ln(\frac{\eta v}{\gamma})}{2(\gamma+\eta)}I_{A}~~{and}~~P(A)=e^{2\eta\left( \pi_{0}-\frac{2\gamma-\ln(\frac{\eta v}{\gamma})}{2(\gamma+\eta)}\right) }.
	\end{align*}
\end{proposition}

\begin{proof}
	According to Proposition \ref{prop-lagran}, $\xi^{\ast}$ is  an optimal strategy if and only if there exists a positive constant $v>0$ such that
	\begin{align*}
	\xi^{\ast}= & I_{\{2\eta ve^{2\eta\xi^{\ast}}-2\gamma e^{2\gamma(1- \xi^{\ast})}<0\}}+bI_{\{2\eta v e^{2\eta\xi^{\ast}}-2\gamma e^{2\gamma(1- \xi^{\ast})}=0\}}
	\end{align*}
	and $\mathbb{E}[e^{2\eta\xi^{\ast}}] = \tilde{\pi}_{0}$, where $b$ is a random variable with $0\leq b\leq1.$ In the following, we can find upper and lower bounds for $v$. If $v<\frac{\gamma}{\eta}e^{-2\eta}$, then for any $\xi\in U$, it holds that
	\begin{align*}
	2\eta v e^{2\eta\xi}-2\gamma e^{2\gamma(1- \xi)}\leq2\eta v e^{2\eta}-2\gamma<0,
	\end{align*}
	and therefore, $\xi^{\ast} =1$, which contradicts  the constraint that $\mathcal{E}_g[\xi^*]\le \pi_{0}<1$. Thus, we must have $v\geq\frac{\gamma}{\eta}e^{-2\eta}$.  Similarly, we can  show that $v\leq\frac{\gamma}{\eta}e^{2\eta}$. Using the fact that $\frac{\gamma}{\eta}e^{-2\eta} \leq v \leq\frac{\gamma}{\eta}e^{2\gamma}$, we can further simplify the optimal randomized test as follows,
		\begin{align*}
	\xi^{\ast}= & b I_{\{\eta v e^{2\eta\xi^{\ast}}-\gamma e^{2\gamma(1- \xi^{\ast})}=0\}}\\
	= &  b I_{\{\eta v e^{2\eta\xi^{\ast}}-\gamma e^{2\gamma(1- \xi^{\ast})}=0\}\cap\{\xi^{\ast}=b \} }\\
	= &  \frac{2\eta-\ln(\frac{\eta v}{\gamma})}{2(\gamma+\eta)} I_{\{\xi^{\ast}=\frac{2\gamma-\ln(\frac{\eta v}{\gamma})}{2(\gamma+\eta)}\}}.
	\end{align*}
	Let $A=\{\omega:\xi^{\ast}=\frac{2\gamma-\ln(\frac{\eta v}{\gamma})}{2(\gamma+\eta)}\}$. It follows from $E[e^{2\eta\xi^{\ast}}] = \tilde{\pi}_{0}$ that
	\[
	P(A)=e^{2\eta\left( \pi_{0}-\frac{2\gamma-\ln(\frac{\eta v}{\gamma})}{2(\gamma+\eta)}\right) },
	\]
	and thus, the proof is complete. 
\end{proof}

	\subsection{Optimization problems under risk measures without cash-additivity}
	
	Recall that a risk measure $\rho(\cdot)$ is cash additive  if 	$$
	\rho(\xi+c)=\rho(\xi)+c, \text{ for } \xi \in L^{2}(\mathcal{F}_{T}) \text{ and } c\in \mathbb{R},
	$$
The assumption of cash additivity, which is also called  cash invariance or translation invariance, 
	is often motivated by the interpretation of
$\rho(\xi)$ as a capital requirement for the financial position $\xi$.
	
	Suppose that $f$ and $g$ are of the form  $f(t,y,z)=a_1y+b_1z^2$, $g(t,y,z)=a_2y+b_2z^2$ for constants $a_i,b_i,  i=1,2$.  Then Proposition 4.3 in \cite{ma} states that   the corresponding $g$-expectations are not cash additive   if $a_i \neq 0,i=1,2$.
	
	By taking $\alpha=0$ , the optimization problem  (\ref{g-object}) becomes
	\begin{equation} \label{eq-non-cash-add}
	\left\{
	\begin{aligned} \inf\limits_{\xi\in U} & \ \mathcal{E}_{f}[h(\xi)]\\ s.t. & \ \mathcal{E}_{g}[\xi]\leq\pi_{0}, \end{aligned} \right.
	\end{equation}
	where   $\pi_{0}$ is a constant,  and the admissible set $U$   is given by
	$$
	U=\{\xi\in L^{\infty}(\mathcal{F}_{T});X\leq\xi\leq Y\}
	$$
	with $X, Y\in L^{\infty}(\mathcal{F}_{T})$ and $X\leq Y$.   
	Under this setting,  the optimization problem related to  the risk measure without cash-additivity has the following result.
	\begin{proposition}
		\label{non-cash-add-pro}
		Suppose that $f$ and $g$ satisfy Assumptions \ref{ass-qua} and \ref{ass-ini} and $h$ obeys Assumption~\ref{ass-h}. If $\xi^{\ast}\in U$ is an optimal solution to (\ref{eq-non-cash-add}), then $\xi^{\ast}$ must be of the following form:
		\begin{align*}
		\xi^{\ast}= & YI_{\{h_{1}h_{x}(\xi^{\ast})n(T)+h_{2}m(T)<0\}}+XI_{\{h_{1}h_{x}(\xi^{\ast})n(T)+h_{2}m(T)>0\}}\\
		& +bI_{\{h_{1}h_{x}(\xi^{\ast})n(T)+h_{2}m(T)=0\}},
		\end{align*}
		where $h_{1}\geq0$ and  $h_{2}\in\mathbb{R}$ satisfy $h_1^2+h_2^2=1$, $b\in L^{\infty}(\mathcal{F}_{T})$ satisfies $X\leq b\leq Y$, and  $m(\cdot)$ and $n(\cdot)$ are given by the following adjoint equations:
		\begin{equation}\label{eq-adjoint-non-cash-add}
		\left\{
		\begin{aligned} dn(t)&=a_1n(t)dt+2b_1z_{1}^{\ast}(t)n(t)dW_{t},\\ dm(t)&=a_2m(t)dt+2b_2z_{2}^{\ast}(t)m(t)dW_{t},\\ n(0)&=1,\ \ m(0)=1. \end{aligned} \right.
		\end{equation}
		
	\end{proposition}

\section*{Acknowledgement}
We are highly grateful to Prof. Chengguo Weng for his helpful suggestions and comments.  The authors moreover thank two anonymous referees for their helpful remarks and comments.

\end{document}